\documentclass{amsart}%
\usepackage{amsmath}
\usepackage{color}
\usepackage{amssymb}
\usepackage{amsfonts}
\usepackage{graphicx}%
\newtheorem{theorem}{Theorem}[section]
\theoremstyle{plain}

\newtheorem{corollary}[theorem]{Corollary}

\newtheorem{lemma}[theorem]{Lemma}

\newtheorem{proposition}[theorem]{Proposition}

\numberwithin{equation}{section}

\begin{document}
\title[Layered solutions above the higher critical exponents]{Positive and nodal single-layered solutions to supercritical elliptic problems
above the higher critical exponents}
\author{M\'{o}nica Clapp}
\address{Instituto de Matem\'{a}ticas, Universidad Nacional Aut\'{o}noma de M\'{e}xico,
Circuito Exterior, C.U., 04510 Mexico City, Mexico}
\email{monica.clapp@im.unam.mx}
\author{Matteo Rizzi}
\address{Instituto de Matem\'{a}ticas, Universidad Nacional Aut\'{o}noma de M\'{e}xico,
Circuito Exterior, C.U., 04510 Mexico City, Mexico}
\email{mrizzi@im.unam.mx}
\thanks{M. Clapp is supported by CONACYT grant 237661 and PAPIIT-DGAPA-UNAM grant
IN104315 (Mexico). M. Rizzi is supported by a postdoctoral fellowship under
CONACYT grant 237661 (Mexico).}
\dedicatory{To Jean Mawhin on his 75th birthday, with great appreciation.}\date{\today}
\maketitle

\begin{abstract}
We study the problem%
\[
-\Delta v+\lambda v=\left\vert v\right\vert ^{p-2}v\text{ in }\Omega
,\text{\qquad}v=0\text{ on $\partial\Omega$},\text{ }%
\]
for $\lambda\in\mathbb{R}$ and supercritical exponents $p,$ in domains of the
form%
\[
\Omega:=\{(y,z)\in\mathbb{R}^{N-m-1}\times\mathbb{R}^{m+1}:(y,\left\vert
z\right\vert )\in\Theta\},
\]
where $m\geq1,$ $N-m\geq3,$ and $\Theta$ is a bounded domain in $\mathbb{R}%
^{N-m}$ whose closure is contained in $\mathbb{R}^{N-m-1}\times(0,\infty)$.
Under some symmetry assumptions on $\Theta$, we show that this problem has
infinitely many solutions for every $\lambda$ in an interval which contains
$[0,\infty)$ and $p>2$ up to some number which is larger than the $(m+1)^{st}$
critical exponent $2_{N,m}^{\ast}:=\frac{2(N-m)}{N-m-2}$. We also exhibit
domains with a shrinking hole, in which there are a positive and a nodal
solution which concentrate on a sphere, developing a single layer that blows
up at an $m$-dimensional sphere contained in the boundary of $\Omega,$ as the
hole shrinks and $p\rightarrow2_{N,m}^{\ast}$ from above. The limit profile of
the positive solution, in the transversal direction to the sphere of
concentration, is a rescaling of the standard bubble, whereas that of the
nodal solution is a rescaling of a nonradial sign-changing solution to the
problem%
\[
-\Delta u=\left\vert u\right\vert ^{2_{n}^{\ast}-2}u,\text{\qquad}u\in
D^{1,2}(\mathbb{R}^{n}),
\]
where $2_{n}^{\ast}:=\frac{2n}{n-2}$ is the critical exponent in dimension
$n.$\medskip



\end{abstract}

\baselineskip15pt

\section{Introduction}

We study the existence and concentration behavior of solutions to the problem
\begin{equation}%
\begin{cases}
-\Delta v+\lambda v=\left\vert v\right\vert ^{p-2}v & \text{in $\Omega$},\\
v=0 & \text{on $\partial\Omega$},
\end{cases}
\tag{$\wp_p$}\label{prob}%
\end{equation}
where $\Omega$ is a bounded domain in $\mathbb{R}^{N},$ $\lambda\in
\mathbb{R},$ and $p$ is supercritical, i.e., it is larger than the critical
Sobolev exponent $2_{N}^{\ast}:=\frac{2N}{N-2}$ for $N\geq3.$ We shall
consider domains of the form
\begin{equation}
\Omega:=\{(y,z)\in\mathbb{R}^{N-m-1}\times\mathbb{R}^{m+1}:(y,\left\vert
z\right\vert )\in\Theta\}, \label{Omega}%
\end{equation}
where $m\geq1,$ $N-m\geq3,$ and $\Theta$ is a bounded domain in $\mathbb{R}%
^{N-m}$ whose closure is contained in $\mathbb{R}^{N-m-1}\times(0,\infty)$.

In domains of this type, the true critical exponent is $2_{N,m}^{\ast}%
:=\frac{2(N-m)}{N-m-2},$ which is the critical Sobolev exponent in the
dimension of $\Theta$ and is larger than $2_{N}^{\ast}.$ Indeed, one can
easily verify that the solutions to the problem (\ref{prob}) which are radial
in the variable $z,$ correspond to the solutions of the problem%
\begin{equation}%
\begin{cases}
-\,\text{div}(f(x)u)+\lambda f(x)u=f(x)\left\vert u\right\vert ^{p-2}u &
\text{in }\Theta,\\
u=0 & \text{on $\partial$}\Theta,
\end{cases}
\label{reduced_prob}%
\end{equation}
where $f(x_{1},...,x_{N-m})=x_{N-m}^{m}.$ Standard variational methods show
that this last problem has infinitely many solutions for $p\in(2,2_{N-m}%
^{\ast}),$ hence, also does the problem (\ref{prob}). On the other hand,
Passaseo showed in \cite{p1,p2} that, if $\lambda=0$ and $\Theta$ is a ball
centered on the half-line $\{0\}\times(0,\infty),$ then the problem
(\ref{prob}) does not have a nontrivial solution for $p\geq2_{N-m}^{\ast
}=2_{N,m}^{\ast}.$ The number $2_{N,m}^{\ast}$ has been called the
$(m+1)^{st}$\textit{ critical exponent} in dimension $N.$

The concentration behavior of solutions to the problem (\ref{prob}) for
$\lambda=0$ and $p\in(2,2_{N,m}^{\ast}),$ as $p\rightarrow2_{N,m}^{\ast}$ from
below, has been investigated in several papers. In \cite{dmp}, del Pino, Musso
and Pacard exhibited positive solutions which concentrate and blow up at a
nondegenerate closed geodesic in $\partial\Omega$, as $p$ approaches the
second critical exponent $2_{N,1}^{\ast}$ from below. For any $m\geq1,$
positive and sign-changing solutions in domains of the form (\ref{Omega}) were
constructed in \cite{acp, kp}. These solutions concentrate and blow up at one
or several $m$-dimensional spheres, as $p\rightarrow2_{N,m}^{\ast}$ from
below. In all of these cases the limit profile of the solutions, in the
transversal direction to each sphere of concentration, is a sum of rescalings
of $\pm U$, where%
\[
U(x):=[n(n-2)]^{(n-2)/4}\left(  \frac{1}{1+|x|^{2}}\right)  ^{(n-2)/2}%
\]
is the standard bubble in dimension $n:=N-m$, which is the only positive
solution to the limit problem%
\begin{equation}
-\Delta u=\left\vert u\right\vert ^{2_{n}^{\ast}-2}u,\text{\qquad}u\in
D^{1,2}(\mathbb{R}^{n}), \label{limprob}%
\end{equation}
up to translation and dilation.

It was recently shown in \cite{c} that there exist nonradial sign-changing
solutions to the problem (\ref{limprob}), that do not resemble a sum of
rescaled positive and negative standard bubbles, which occur as limit profiles
for concentration of sign-changing solutions to the problem (\ref{prob}) that
blow up at a single point, as $p\rightarrow2_{N}^{\ast}$ from below. For the
higher critical exponents $2_{N,m}^{\ast}$ with $m\geq1,$ it was shown in
\cite{cf} that for every $\lambda$ in some interval which contains
$[0,\infty)$ there are sign-changing solutions to the problem (\ref{prob}), in
domains of the form (\ref{Omega}), which concentrate and blow up at an
$m$-dimensional sphere, as $p\rightarrow2_{N,m}^{\ast}$ from below, whose
limit profile in the transversal direction to the sphere of concentration is a
nonradial sign-changing solution to (\ref{limprob}), like those found in
\cite{c}.

The study of concentration phenomena for $p$ approaching $2_{N}^{\ast}$ from
above, is a much more delicate issue, beginning with the fact that solutions
to (\ref{prob}) for $p>2_{N}^{\ast}$ do not always exist. For $\lambda=0,$
standard bubbles were used as basic cells in \cite{dfm,dfm2,mpi,pr} to
construct positive solutions to the slightly supercritical problem
(\ref{prob}) with $p=2_{N}^{\ast}+\varepsilon,$ for small enough
$\varepsilon>0$, in domains with a hole, using the Ljapunov-Schmidt reduction
method. These solutions blow up, as $\varepsilon\rightarrow0,$ and their limit
profile at each blow-up point is a rescaling of the standard bubble. Solutions
in some contractible domains were constructed in \cite{mp1,mp2}.

Quite recently, sign-changing solutions to the slightly supercritical problem
(\ref{prob}) with $p=2_{N}^{\ast}+\varepsilon,$ $\varepsilon>0,$ were
exhibited by Musso and Wei \cite{mw} in domains with a small fixed hole, and
by Clapp and Pacella \cite{cp} in domains with a shrinking hole. The solutions
obtained in \cite{mw} concentrate at two different points in the domain, as
$\varepsilon\rightarrow0,$ and their limit profile at each of them is a
rescaling of one of the sign-changing solutions to the limit problem
(\ref{limprob}) in $\mathbb{R}^{N}$ constructed by del Pino, Musso, Pacard and
Pistoia in \cite{dmpp}, which resemble a large number of negative bubbles,
placed evenly along a circle, surrounding a positive bubble, placed at its
center. On the other hand, the sign-changing solutions exhibited in \cite{cp}
concentrate at a single point in the interior of the shrinking hole, as the
hole shrinks and $\varepsilon\rightarrow0,$ and their limit profile is a
rescaling of a nonradial sign-changing solution to (\ref{limprob}) like those
found in \cite{c}.

For $m\geq1,$ the existence of solutions for $p=2_{N,m}^{\ast}+\varepsilon$
and their concentration behavior seems to be, so far, an open question; see
Problem 4 in \cite{cpi}.\ In this paper we will show that, under some symmetry
assumptions, the problem (\ref{prob}) has infinitely many solutions in domains
of the form (\ref{Omega}) for $p>2_{N,m}^{\ast}$, up to some value which
depends on the symmetries; see Theorem \ref{thm:main_existence}. We will also
exhibit domains with a shrinking hole, in which there are positive and
sign-changing solutions which concentrate and blow up at an $m$-dimensional
sphere contained in the boundary of $\Omega,$ as the hole shrinks and
$p\rightarrow2_{N,m}^{\ast}$ from above. The limit profile of the positive
solutions, in the direction transversal to the sphere of concentration, will
be a rescaling of the standard bubble, whereas that of the sign-changing ones
will resemble one of the solutions to (\ref{limprob}) that were found in
\cite{c}.

We give, next, some examples of our results. For $n:=N-m,$ let $B$ be an
$n$-dimensional ball of radius $\delta_{0},$ centered on the half-line
$\{0\}\times(0,\infty),$ whose closure is contained in the half-space
$\mathbb{R}^{n-1}\times(0,\infty).$ We write the points in $\mathbb{R}%
^{n-1}\times(0,\infty)$ as $(y,t)$ with $y\in\mathbb{R}^{n-1},$ $t\in
(0,\infty)$\ and we set%
\begin{align*}
B_{\delta}  &  :=\{(y,t)\in B:\left\vert y\right\vert >\delta\}\text{ \ \ if
}\delta\in(0,\delta_{0}),\qquad B_{0}:=B,\\
\Omega_{\delta}  &  :=\{(y,z)\in\mathbb{R}^{n-1}\times\mathbb{R}%
^{m+1}:(y,\left\vert z\right\vert )\in B_{\delta}\},\qquad\Omega:=\Omega_{0}.
\end{align*}
We denote by $O(k)$ the group of all linear isometries of $\mathbb{R}^{k}$
and, for $v\in D^{1,2}(\mathbb{R}^{N}),$ we write
\[
\left\Vert v\right\Vert :=\left(  \int_{\mathbb{R}^{N}}\left\vert \nabla
v\right\vert ^{2}\right)  ^{1/2}.
\]
The following results establish the existence of positive and sign-changing
solutions to the problem (\ref{prob}) in $\Omega_{\delta}\ $and describe their
limit profile as $\delta\rightarrow0$ and $p\rightarrow2_{N,m}^{\ast}$ from
above. They are special cases of Theorems \ref{thm:main_existence} and
\ref{thm:supercritical_profile}, which apply to more general domains, and are
stated and proved in Sections \ref{sec:existence} and \ref{sec:blow-up}, respectively.

\begin{theorem}
\label{thm:main1}There exists $\lambda_{\ast}\leq0$ such that, for each
$\lambda\in(\lambda_{\ast},\infty)\cup\{0\},$ $\delta\in(0,\delta_{0})$ and
$p\in(2,\infty),$ the problem \emph{(\ref{prob})} has a positive solution
$v_{\delta,p}$ in $\Omega_{\delta}$ which satisfies%
\[
v_{\delta,p}(\gamma y,\varrho z)=v_{\delta,p}(y,z)\qquad\forall\gamma\in
O(n-1),\text{ }\varrho\in O(m+1),\ (y,z)\in\Omega_{\delta},
\]
and has minimal energy among all nontrivial solutions to \emph{(\ref{prob})}
in $\Omega_{\delta}$ with these symmetries.

Moreover, there exist sequences $(\delta_{k})$ in $(0,\delta_{0}),$ $(p_{k})$
in $(2_{N,m}^{\ast},\infty),$ $(\varepsilon_{k})$ in $(0,\infty)$ and
$(\zeta_{k})$ in $B\cap\left[  \{0\}\times(0,\infty)\right]  $ such that

\begin{enumerate}
\item[(i)] $\delta_{k}\rightarrow0,$ $\ p_{k}\rightarrow2_{N,m}^{\ast},$
$\ \varepsilon_{k}^{-1}$\emph{dist}$(\zeta_{k},\partial\Theta)\rightarrow
\infty,$ and $\zeta_{k}\rightarrow\zeta$ with%
\[
\text{\emph{dist}}(\zeta,\mathbb{R}^{n-1}\times\{0\})=\text{\emph{dist}%
}(B,\mathbb{R}^{n-1}\times\{0\}),
\]

\item[(ii)] $\lim_{k\rightarrow\infty}\left\Vert v_{\delta_{k},p_{k}%
}-\widetilde{U}_{\varepsilon_{k},\zeta_{k}}\right\Vert =0,$ where%
\[
\widetilde{U}_{\varepsilon_{k},\zeta_{k}}(y,z):=\varepsilon_{k}^{(2-n)/2}%
U\left(  \frac{(y,\left\vert z\right\vert )-\zeta_{k}}{\varepsilon_{k}%
}\right)
\]
and $U$ is the standard bubble in dimension $n.$
\end{enumerate}

The number $\lambda_{\ast}$ is negative if $m\geq2.$
\end{theorem}

The solutions given by Theorem \ref{thm:main1}\ concentrate on an
$m$-dimensional sphere, developing a positive layer which blows up at an
$m$-dimensional sphere contained in the boundary of $\Omega$ and located at
minimal distance to the plane of rotation $\mathbb{R}^{n-1}\times\{0\}.$ The
asymptotic profile of each layer in the transversal direction to its sphere of
concentration is a rescaling of the standard bubble.

The next theorem gives sign-changing solutions to the problem (\ref{prob})
with a different type of asymptotic profile. For $n\geq5$ and we write
$\mathbb{R}^{n-1}\equiv\mathbb{C}^{2}\times\mathbb{R}^{n-5}$ and the points in
$\mathbb{R}^{n-1}$ as $y=(\eta,\xi)$ with $\eta=(\eta_{1},\eta_{2}%
)\in\mathbb{C}^{2},$ $\xi\in\mathbb{R}^{n-5}.$

\begin{theorem}
\label{thm:main2}Assume that $n=5$ or $n\geq7.$ Then, there exists
$\lambda_{\ast}\leq0$ such that, for each $\lambda\in(\lambda_{\ast}%
,\infty)\cup\{0\},$ $\delta\in(0,\delta_{0})$ and $p\in(2,2_{N,m+1}^{\ast}),$
the problem \emph{(\ref{prob})} has a nontrivial sign-changing solution
$w_{\delta,p}$ in $\Omega_{\delta}$ which satisfies%
\[
w_{\delta,p}(\eta,\xi,z)=w_{\delta,p}(\mathrm{e}^{\mathrm{i}\vartheta}%
\eta,\alpha\xi,\varrho z),\qquad w_{\delta,p}(\eta_{1},\eta_{2},\xi
,z)=-w_{\delta,p}(-\bar{\eta}_{2},\bar{\eta}_{1},\xi,z),
\]
for every $\vartheta\in\lbrack0,\pi),$ $\alpha\in O(n-5),$ $\varrho\in O(m+1)$
and\ $(y,z)\in\Omega_{\delta},$ and which has minimal energy among all
nontrivial solutions to \emph{(\ref{prob})} in $\Omega_{\delta}$ with these
symmetry properties.

Moreover, there exist sequences $(\delta_{k})$ in $(0,\delta_{0}),$ $(p_{k})$
in $(2_{N,m}^{\ast},2_{N,m+1}^{\ast}),$ $(\varepsilon_{k})$ in $(0,\infty)$
and $(\zeta_{k})$ in $B\cap\left[  \{0\}\times(0,\infty)\right]  ,$ and a
nontrivial sign-changing solution $W$ to the limit problem
\emph{(\ref{limprob}),} such that

\begin{enumerate}
\item[(i)] $\delta_{k}\rightarrow0,$ $\ p_{k}\rightarrow2_{N,m}^{\ast},$
$\ \varepsilon_{k}^{-1}$\emph{dist}$(\zeta_{k},\partial\Theta)\rightarrow
\infty,$ and $\ \zeta_{k}\rightarrow\zeta$ \ with
\[
\text{\emph{dist}}(\zeta,\mathbb{R}^{n-1}\times\{0\})=\text{\emph{dist}%
}(\Theta,\mathbb{R}^{n-1}\times\{0\}),
\]

\item[(ii)] $W(\eta,\xi,t)=W(\mathrm{e}^{\mathrm{i}\vartheta}\eta,\alpha
\xi,t)$ and $W(\eta_{1},\eta_{2},\xi,t)=-W(-\bar{\eta}_{2},\bar{\eta}_{1}%
,\xi,t)$ for every $\vartheta\in\lbrack0,\pi),$ $\alpha\in O(n-5)$
and\ $(y,t)\in\mathbb{R}^{n-1}\times\mathbb{R}\equiv\mathbb{R}^{n},$ and $W$
has minimal energy among all nontrivial solutions to \emph{(\ref{limprob})}
with these symmetry properties,

\item[(iii)] $\lim_{k\rightarrow\infty}\left\Vert w_{\delta_{k},p_{k}%
}-\widetilde{W}_{\varepsilon_{k},\zeta_{k}}\right\Vert =0,$ where%
\[
\widetilde{W}_{\varepsilon_{k},\zeta_{k}}(y,z):=\varepsilon_{k}^{(2-n)/2}%
W\left(  \frac{(y,\left\vert z\right\vert )-\zeta_{k}}{\varepsilon_{k}%
}\right)  .
\]

\end{enumerate}

The number $\lambda_{\ast}$ is negative if $m\geq2.$
\end{theorem}

The solutions given by Theorem \ref{thm:main2}\ concentrate on an
$m$-dimensional sphere, developing a sign-changing layer which blows up at an
$m$-dimensional sphere contained in the boundary of $\Omega$ and located at
minimal distance to the plane of rotation $\mathbb{R}^{n-1}\times\{0\}.$ The
asymptotic profile of each layer in the transversal direction to its sphere of
concentration is a rescaling of a nonradial sign-changing solution to the
limit problem (\ref{limprob}), like those found in \cite{c}.

As we mentioned before, the solutions to the anisotropic problem
(\ref{reduced_prob}) give rise to solutions of the problem (\ref{prob}) in
domains of the form (\ref{Omega}). In Section \ref{sec:existence} we will
study a general anisotropic problem in an $n$-dimensional domain $\Theta.$ We
will assume that $\Theta$ has some symmetries and we will establish the
existence of infinitely many positive and sign-changing solutions to the
anisotropic problem for supercritical exponents $p>2_{n}^{\ast},$ up to some
value which depends on the symmetries. These results extend those obtained in
\cite{cp} for the problem with constant coefficients. In Section
\ref{sec:min_crit} we will describe the behavior of the minimizing sequences
for the variational functional associated to the anisotropic problem for
$p=2_{n}^{\ast}$. These sequences, either converge to a solution, or they blow
up. We will provide information on the location of the blow-up points and on
the symmetries of the solutions to the limit problem (\ref{limprob})\ which
occur as limit profiles. This will be used in Section \ref{sec:blow-up} to
obtain information on the concentration behavior and the limit profile of
positive and sign-changing solutions to the problem (\ref{prob}) in domains
with a shrinking hole, as the hole shrinks and $p\rightarrow2_{N,m}^{\ast}$
from above.

\section{Symmetries and existence for supercritical problems}

\label{sec:existence}Let $\Gamma$ be a closed subgroup of $O(n)$ and
$\phi:\Gamma\rightarrow\mathbb{Z}_{2}$ be a continuous homomorphism of groups.
A function $u:\mathbb{R}^{n}\rightarrow\mathbb{R}$ is said to be $\phi
$\emph{-equivariant} if
\begin{equation}
u(\gamma x)=\phi(\gamma)u(x)\qquad\forall\gamma\in\Gamma,\text{ }%
x\in\mathbb{R}^{n}. \label{eq:symm}%
\end{equation}
If $\phi$ is the trivial homomorphism, then (\ref{eq:symm}) simply says that
$u$ is a $\Gamma$-invariant function, whereas, if $\phi$ is surjective and $u$
is not trivial, then (\ref{eq:symm}) implies that $u$ is sign-changing,
nonradial and $G$-invariant, where $G:=\ker\phi.$

Let $\Theta$ be a bounded $\Gamma$-invariant domain in $\mathbb{R}^{n},$
$n\geq3,$ and $a\in\mathcal{C}^{1}(\overline{\Theta}),$ $b,c\in\mathcal{C}%
^{0}(\overline{\Theta})$ be $\Gamma$-invariant functions satisfying $a,c>0$ on
$\overline{\Theta}$. We assume that%
\begin{equation}
\text{there exists }x_{0}\in\Theta\text{ such that }\{\gamma\in\Gamma:\gamma
x_{0}=x_{0}\}\subset\ker\,\phi. \label{A}%
\end{equation}
This assumption guarantees that the space
\[
D_{0}^{1,2}(\Theta)^{\phi}:=\{u\in D_{0}^{1,2}(\Theta):u\text{ is }%
\phi\text{-equivariant}\}
\]
is infinite dimensional; see \cite{bcm}. As usual, $D_{0}^{1,2}(\Theta)$
denotes the closure of $\mathcal{C}_{c}^{\infty}(\Theta)$ in the Hilbert space%
\[
D^{1,2}(\mathbb{R}^{n}):=\{u\in L^{2_{n}^{\ast}}(\mathbb{R}^{n}):\nabla u\in
L^{2}(\mathbb{R}^{n},\mathbb{R}^{n})\}
\]
equiped with the norm%
\[
\left\Vert u\right\Vert :=\left(  \int_{\Theta}\left\vert \nabla u\right\vert
^{2}\right)  ^{1/2}.
\]
We shall also assume that the operator $-\,$div$(a\nabla)+b$ is coercive in
$D_{0}^{1,2}(\Theta)^{\phi},$ i.e., that%
\begin{equation}
\inf_{\substack{u\in D_{0}^{1,2}(\Theta)^{\phi}\\u\neq0}}\frac{\int_{\Theta
}(a(x)\left\vert \nabla u\right\vert ^{2}+b(x)u^{2})\mathrm{d}x}{\int_{\Theta
}\left\vert \nabla u\right\vert ^{2}}>0. \label{eq:coercivity}%
\end{equation}
We set
\[
\left\Vert u\right\Vert _{a,b}^{2}:=\int_{\Theta}(a(x)\left\vert \nabla
u\right\vert ^{2}+b(x)u^{2})\mathrm{d}x,\qquad\left\vert u\right\vert
_{c;p}^{p}:=\int_{\Theta}c(x)\left\vert u\right\vert ^{p}\mathrm{d}x.
\]
Assumption (\ref{eq:coercivity}) implies that $\left\Vert \cdot\right\Vert
_{a,b}$ is a norm in $D_{0}^{1,2}(\Theta)^{\phi}$ which is equivalent
$\left\Vert \cdot\right\Vert .$ Note that, as $c>0,$ $\left\vert
\cdot\right\vert _{c;p}$ is equivalent to the standard norm in $L^{p}%
(\Theta),$ which we denote by $\left\vert \cdot\right\vert _{p}.$

Our aim is to establish the existence of solutions to the problem
\begin{equation}%
\begin{cases}
-\,\text{div}(a(x)\nabla u)+b(x)u=c(x)|u|^{p-2}u & \text{in $\Theta$,}\\
u=0 & \text{on $\partial\Theta$.}\\
u(\gamma x)=\phi(\gamma)u(x), & \forall\gamma\in\Gamma\text{, \ }x\in\Theta,
\end{cases}
\label{anisotropic_prob}%
\end{equation}
for every $2<p<2_{n-d}^{\ast}$, where
\[
d:=\min\{\dim(\Gamma x):x\in\overline{\Theta}\},
\]
$\Gamma x:=\{\gamma x:\gamma\in\Gamma\}$ is the $\Gamma$-orbit of $x,$
$2_{k}^{\ast}:=\frac{2k}{k-2}$ if $k\geq3$ and $2_{k}^{\ast}:=\infty$ if
$k=1,2.$ Note that $2_{n-d}^{\ast}>2_{n}^{\ast}$ if $d>0.$

A (weak) solution to the problem (\ref{anisotropic_prob}) is a function $u\in
D_{0}^{1,2}(\Theta)^{\phi}\cap L^{p}(\Theta)$ such that%
\begin{equation}
\int_{\Theta}(a(x)\nabla u\cdot\nabla\psi+b(x)u\psi)\mathrm{d}x-\int_{\Theta
}c(x)|u|^{p-2}u\psi\,\mathrm{d}x=0\text{\qquad}\forall\psi\in\mathcal{C}%
_{c}^{\infty}(\Theta). \label{eq:solution}%
\end{equation}
Proposition 2.1 of \cite{cp} asserts that $D_{0}^{1,2}(\Theta)^{\phi}$ is
continuously embedded in $L^{p}(\Theta)$ for any real number $p\in
\lbrack1,2_{n-d}^{\ast}]$, and that the embedding is compact for $p\in
\lbrack1,2_{n-d}^{\ast})$. The proof relies on a result by Hebey and Vaugon
\cite{hv} which establishes these facts for $\Gamma$-invariant functions.
Therefore, the functional%
\[
J_{p}(u):=\frac{1}{2}\left\Vert u\right\Vert _{a,b}^{2}-\frac{1}{p}\left\vert
u\right\vert _{c;p}^{p}%
\]
is well defined in the space $D_{0}^{1,2}(\Theta)^{\phi}$ if $p\in
(2,2_{n-d}^{\ast}].$

\begin{lemma}
\label{lemma_crit}For any real number $p\in(2,2_{n-d}^{\ast}]$ the critical
points of the functional $J_{p}$ in the space $D_{0}^{1,2}(\Theta)^{\phi}$ are
the solutions to the problem \emph{(\ref{anisotropic_prob})}.
\end{lemma}

\begin{proof}
Let $u\in D_{0}^{1,2}(\Theta)^{\phi}$ be a critical point of $J_{p}$ in
$D_{0}^{1,2}(\Theta)^{\phi}$. Then,
\[
J_{p}^{\prime}(u)\vartheta=\int_{\Theta}(a(x)\nabla u\cdot\nabla
\vartheta+b(x)u\vartheta-c(x)|u|^{p-2}u\vartheta)\,\mathrm{d}x=0\text{\quad
}\forall\vartheta\in D_{0}^{1,2}(\Theta)^{\phi}.
\]
As $D_{0}^{1,2}(\Theta)^{\phi}\subset L^{p}(\Theta)$ for $1\leq p\leq
2_{n-d}^{\ast}$ we need only to prove that $u$ satisfies (\ref{eq:solution}).
Let $\psi\in\mathcal{C}_{c}^{\infty}(\Theta),$ and define
\[
\widetilde{\psi}(x):=\frac{1}{\mu(\Gamma)}\int_{\Gamma}\phi(\gamma)\psi(\gamma
x)\mathrm{d}\mu,
\]
where $\mu$ is the Haar measure on $\Gamma$. Note that $\widetilde{\psi}\in
D_{0}^{1,2}(\Theta)^{\phi}.$ Observe also that, as $u$ is $\phi$-equivariant,
we have that%
\[
\phi(\gamma)\nabla u(x)=\nabla\left(  u\circ\gamma\right)  (x)=\gamma
^{-1}\nabla u(\gamma x)\qquad\forall\gamma\in\Gamma,\text{ }x\in\Theta.
\]
Since $J_{p}^{\prime}(u)\widetilde{\psi}=0,$ and $a,b,c$ are $\Gamma
$-invariant, using Fubini's theorem and performing a change of variable, we
get%
\begin{align*}
0  &  =\int_{\Theta}(a(x)\nabla u(x)\cdot\nabla\widetilde{\psi}%
(x)+b(x)u(x)\widetilde{\psi}(x)-c(x)|u(x)|^{p-2}u(x)\widetilde{\psi
}(x))\mathrm{d}x\\
&  =\frac{1}{\mu(\Gamma)}\int_{\Theta}\int_{\Gamma}\left[  a(x)\phi
(\gamma)\nabla u(x)\cdot\gamma^{-1}\nabla\psi(\gamma x)+b(x)\phi
(\gamma)u(x)\psi(\gamma x)\right. \\
&  \hspace{1in}\left.  -c(x)|\phi(\gamma)u(x)|^{p-2}\phi(\gamma)u(x)\psi
(\gamma x)\right]  \mathrm{d}\mu\,\mathrm{d}x\\
&  =\frac{1}{\mu(\Gamma)}\int_{\Gamma}\int_{\Theta}\left[  a(x)\gamma
^{-1}\nabla u(\gamma x)\cdot\gamma^{-1}\nabla\psi(\gamma x)+b(x)u(\gamma
x)\psi(\gamma x)\right. \\
&  \hspace{1in}\left.  -c(x)|u(\gamma x)|^{p-2}u(\gamma x)\psi(\gamma
x)\right]  \mathrm{d}x\,\mathrm{d}\mu\\
&  =\frac{1}{\mu(\Gamma)}\int_{\Gamma}\int_{\Theta}\left[  a(\gamma x)\nabla
u(\gamma x)\cdot\nabla\psi(\gamma x)+b(\gamma x)u(\gamma x)\psi(\gamma
x)\right. \\
&  \hspace{1in}\left.  -c(\gamma x)|u(\gamma x)|^{p-2}u(\gamma x)\psi(\gamma
x)\right]  \mathrm{d}x\,\mathrm{d}\mu\\
&  =\frac{1}{\mu(\Gamma)}\int_{\Gamma}\mathrm{d}\mu\int_{\Theta}\left[
a(\xi)\nabla u(\xi)\cdot\nabla\psi(\xi)+b(\xi)u(\xi)\psi(\xi)-c(\xi
)|u(x)|^{p-2}u(\xi)\psi(\xi)\right]  \mathrm{d}\xi\\
&  =\int_{\Theta}\left[  a(\xi)\nabla u(\xi)\cdot\nabla\psi(\xi)+b(\xi
)u(\xi)\psi(\xi)-c(\xi)|u(x)|^{p-2}u(\xi)\psi(\xi)\right]  \mathrm{d}\xi.
\end{align*}
Therefore $u$ is a solution to the problem (\ref{anisotropic_prob}).
\end{proof}

The nontrivial critical points of the functional $J_{p}:D_{0}^{1,2}%
(\Theta)^{\phi}\rightarrow\mathbb{R}$ lie on the Nehari manifold%
\[
\mathcal{N}_{p}^{\phi}:=\left\{  u\in D_{0}^{1,2}(\Theta)^{\phi}:\left\Vert
u\right\Vert _{a,b}^{2}=\left\vert u\right\vert _{c;p}^{p},\,u\neq0\right\}
,
\]
which is a $\mathcal{C}^{2}$-Hilbert manifold, radially diffeomorphic to the
unit sphere in $D_{0}^{1,2}(\Theta)^{\phi},$ and a natural constraint for this
functional. Set
\[
\ell_{p}^{\phi}:=\inf\{J_{p}(u):u\in\mathcal{N}_{p}^{\phi}\}.
\]
Then, $\ell_{p}^{\phi}>0.$ A \textit{least energy solution} to the problem
(\ref{anisotropic_prob}) is a minimizer for $J_{p}$ on $\mathcal{N}_{p}^{\phi
}.$ The following result extends Theorem 2.3 in \cite{cp}.

\begin{theorem}
\label{thm:existence}If $p\in(2,2_{n-d}^{\ast})$ then the problem
\emph{(\ref{anisotropic_prob})} has a least energy solution, and an unbounded
sequence of solutions.
\end{theorem}

\begin{proof}
By Lemma \ref{lemma_crit}, the critical points of the functional $J_{p}$ i n
the space $D_{0}^{1,2}(\Theta)^{\phi}$ are the solutions to the problem
(\ref{anisotropic_prob}). Proposition 2.1 of \cite{cp} asserts that
$D_{0}^{1,2}(\Theta)^{\phi}$ is compactly embedded in $L^{p}(\Theta)$ for
$p\in(2,2_{n-d}^{\ast}),$ hence, a standard argument shows that the functional
$J_{p}:D_{0}^{1,2}(\Theta)^{\phi}\rightarrow\mathbb{R}$ satisfies the
Palais-Smale condition. Therefore, $J_{p}$ attains its minimum on
$\mathcal{N}_{p}^{\phi}.$ Moreover, as the functional is even and has the
mountain pass geometry, the symmetric mountain pass theorem \cite{ar} yields
the existence of an unbounded sequence of critical values for $J_{p}$ on
$D_{0}^{1,2}(\Theta)^{\phi}.$
\end{proof}

We now derive a multiplicity result for the supercritical problem
(\ref{prob}). Assume that the closure of $\Theta$ is contained in
$\mathbb{R}^{n-1}\times(0,\infty)$ and, for $m\geq1,$ let%
\begin{equation}
\lambda_{1}^{\phi}:=\inf_{\substack{u\in D_{0}^{1,2}(\Theta)^{\phi}\\u\neq
0}}\frac{\int_{\Theta}x_{n}^{m}\left\vert \nabla u\right\vert ^{2}}%
{\int_{\Theta}x_{n}^{m}u^{2}}. \label{eq:lambda1}%
\end{equation}
As the $n$-th coordinate $x_{n}$ of $x$ is positive for every $x\in
\overline{\Theta},$ from the Poincar\'{e} inequality we obtain that
$\lambda_{1}^{\phi}>0.$

\begin{theorem}
\label{thm:main_existence}If $\lambda\in(-\lambda_{1}^{\phi},\infty)$ and
$p\in(2,2_{n-d}^{\ast}),$ then the problem \emph{(\ref{prob})} has a least
energy solution and an unbounded sequence of solutions in%
\[
\Omega:=\{(y,z)\in\mathbb{R}^{n-1}\times\mathbb{R}^{m+1}:(y,\left\vert
z\right\vert )\in\Theta\},
\]
which satisfy%
\begin{equation}
v(\gamma y,\varrho z)=\phi(\gamma)v(y,z)\qquad\forall\gamma\in\Gamma,\text{
}\varrho\in O(m+1),\text{\ }(y,z)\in\Omega. \label{rotational_symm}%
\end{equation}

\end{theorem}

\begin{proof}
A straighforward computation shows that $v$ is a solution to the problem
(\ref{prob}) in $\Omega$ which satisfies (\ref{rotational_symm}) if and only
if the function\ $u$ given by $v(y,z)=u(y,\left\vert z\right\vert )$ is a
solution to the problem (\ref{anisotropic_prob}) with $a(x):=x_{n}^{m}=:c(x)$
and $b(x):=\lambda x_{n}^{m}.$ Moreover, $v$ has minimal energy if and only if
$u$ does. Note that (\ref{eq:coercivity}) is satisfied if $\lambda\in
(-\lambda_{1}^{\phi},\infty).$ So this result follows from Theorem
\ref{thm:existence}.
\end{proof}

For $p\in(2,2_{n-d}^{\ast})$ let $u_{p}$ be a least energy solution to the
problem (\ref{anisotropic_prob}). Fix $q\in(2,2_{n-d}^{\ast})$ and let
$t_{q,p}\in(0,\infty)$ be such that $\widetilde{u}_{p}:=t_{q,p}u_{p}%
\in\mathcal{N}_{q}^{\phi},$ i.e.,
\begin{equation}
t_{q,p}=\left(  \frac{\left\Vert u_{p}\right\Vert _{a,b}^{2}}{\left\vert
u_{p}\right\vert _{c;q}^{q}}\right)  ^{\frac{1}{q-2}}=\left(  \frac{\left\vert
u_{p}\right\vert _{c;p}^{p}}{\left\vert u_{p}\right\vert _{c;q}^{q}}\right)
^{\frac{1}{q-2}}. \label{eq:t}%
\end{equation}
We will show that $\lim_{p\rightarrow q}J_{q}\left(  \widetilde{u}_{p}\right)
=\ell_{q}^{\phi}.$ The proof is similar to that of Proposition 2.5\ in
\cite{cp}. We give the details for the reader's convenience, starting with the
following lemma.

\begin{lemma}
\label{lem:p_to_q}If $p_{k},q\in(2,2_{n-d}^{\ast}),$ $p_{k}\rightarrow q,$ and
$(u_{k})$ is a bounded sequence in $D_{0}^{1,2}(\Theta)^{\phi},$ then%
\[
\lim_{k\rightarrow\infty}\int_{\Theta}\left(  c(x)\left\vert u_{k}\right\vert
^{p_{k}}-c(x)\left\vert u_{k}\right\vert ^{q}\right)  \mathrm{d}x=0.
\]

\end{lemma}

\begin{proof}
By the mean value theorem, for each $x\in\Theta,$ there exists $q_{k}(x)$
between $p_{k}$ and $q$ such that%
\[
\left\vert \left\vert u_{k}(x)\right\vert ^{p_{k}}-\left\vert u_{k}%
(x)\right\vert ^{q}\right\vert =\left\vert \ln\left\vert u_{k}(x)\right\vert
\right\vert \,\left\vert u_{k}(x)\right\vert ^{q_{k}(x)}\left\vert
p_{k}-q\right\vert .
\]
Fix $r>0$ such that $\left[  q-r,q+r\right]  \subset(2,2_{n-d}^{\ast}).$ Then,
for some positive constant $C$ and $k$ large enough,%
\[
\left\vert \ln\left\vert u_{k}\right\vert \right\vert \,\left\vert
u_{k}\right\vert ^{q_{k}}\leq\left\{
\begin{array}
[c]{lll}%
\ln\left\vert u_{k}\right\vert \,\left\vert u_{k}\right\vert ^{q+r} & \leq
C\left\vert u_{k}\right\vert ^{2_{n-d}^{\ast}} & \text{if }\left\vert
u_{k}\right\vert \geq1,\\
\left(  \ln\frac{1}{\left\vert u_{k}\right\vert }\right)  \,\left\vert
u_{k}\right\vert ^{q-r} & \leq C\left\vert u_{k}\right\vert ^{2} & \text{if
}\left\vert u_{k}\right\vert \leq1.
\end{array}
\right.
\]
As $D_{0}^{1,2}(\Theta)^{\phi}$ is continuously embedded in $L^{p}(\Theta)$
for $p\in\lbrack2,2_{n-d}^{\ast}],$ we obtain
\begin{align*}
\left\vert \int_{\Theta}c\left(  \left\vert u_{k}\right\vert ^{p_{k}%
}-\left\vert u_{k}\right\vert ^{q}\right)  \right\vert  &  \leq\left\vert
c\right\vert _{\infty}\left(  \int_{\left\vert u_{k}\right\vert \leq
1}\left\vert \left\vert u_{k}\right\vert ^{p_{k}}-\left\vert u_{k}\right\vert
^{q}\right\vert +\int_{\left\vert u_{k}\right\vert >1}\left\vert \left\vert
u_{k}\right\vert ^{p_{k}}-\left\vert u_{k}\right\vert ^{q}\right\vert \right)
\\
&  \leq\left\vert c\right\vert _{\infty}C\left\vert p_{k}-q\right\vert
\int_{\Theta}\left(  \left\vert u_{k}\right\vert ^{2}+\left\vert
u_{k}\right\vert ^{2_{n-d}^{\ast}}\right) \\
&  \leq\bar{C}\left\vert p_{k}-q\right\vert \left\Vert u_{k}\right\Vert
^{2_{n-d}^{\ast}}%
\end{align*}
for some positive constant $\bar{C},$ where $\left\vert c\right\vert _{\infty
}:=\sup_{x\in\Theta}\left\vert c(x)\right\vert .$ Since $(u_{k})$ is bounded
in $D_{0}^{1,2}(\Theta),$ our claim follows.
\end{proof}

\begin{proposition}
\label{prop:limits}For $q\in(2,2_{n-d}^{\ast})$ we have that%
\[
\lim_{p\rightarrow q}\ell_{p}^{\phi}=\ell_{q}^{\phi},\text{\qquad}%
\lim_{p\rightarrow q}t_{q,p}=1,\text{\qquad}\lim_{p\rightarrow q}J_{q}\left(
\widetilde{u}_{p}\right)  =\ell_{q}^{\phi}.
\]

\end{proposition}

\begin{proof}
Set $S_{p}^{\phi}:=\inf_{u\in D_{0}^{1,2}(\Omega)^{\phi}\smallsetminus
\{0\}}\frac{\left\Vert u\right\Vert _{a,b}^{2}}{\left\vert u\right\vert
_{c;p}^{2}}$. It is easy to see that $\ell_{p}^{\phi}=\frac{p-2}{2p}\left(
S_{p}^{\phi}\right)  ^{\frac{p}{p-2}}.$ So, to prove the first identity, it
suffices to show that $\lim_{p\rightarrow q}S_{p}^{\phi}=S_{q}^{\phi}.$ From
H\"{o}lder's inequality we get that $\left\vert u\right\vert _{c;q}%
\leq\left\vert c\right\vert _{1}^{(p-q)/pq}\left\vert u\right\vert _{c;p}$ if
$p>q.$ Hence, $S_{q}^{\phi}\geq\left\vert c\right\vert _{1}^{2(q-p)/pq}%
S_{p}^{\phi}$ if $p>q.$ So, as $p$ approaches $q$ from the right, we have
that
\[
\limsup_{p\rightarrow q^{+}}S_{p}^{\phi}\leq S_{q}^{\phi}.
\]
Assume that $\liminf_{p\rightarrow q^{+}}S_{p}^{\phi}<S_{q}^{\phi}.$ Then,
there exist $\varepsilon>0$ and sequences $\left(  p_{k}\right)  $ in
$(q,2_{n-d}^{\ast})$ and $\left(  u_{k}\right)  $ in $D_{0}^{1,2}%
(\Omega)^{\phi}$ with $\left\vert u_{k}\right\vert _{c;p_{k}}=1$ such that
$\left\Vert u_{k}\right\Vert _{a,b}^{2}<S_{q}^{\phi}-\varepsilon.$ Lemma
\ref{lem:p_to_q} implies that $\frac{\left\Vert u_{k}\right\Vert _{a,b}^{2}%
}{\left\vert u_{k}\right\vert _{c;q}^{2}}<S_{q}^{\phi}$ for $k$ large enough,
contradicting the definition of $S_{q}^{\phi}$. This proves that%
\[
\lim_{p\rightarrow q^{+}}S_{p}^{\phi}=S_{q}^{\phi}.
\]
The corresponding statement when $p$ approaches $q$ from the left is proved in
a similar way. Since $J_{p}(u_{p})=\frac{p-2}{2p}\left\Vert u_{p}\right\Vert
_{a,b}^{2}=\ell_{p}^{\phi}$ we have that $(u_{p})$ is bounded in $D_{0}%
^{1,2}(\Omega)^{\phi}$ for $p$ close to $q.$ Lemma \ref{lem:p_to_q} applied to
(\ref{eq:t}) yields $\lim_{p\rightarrow q}t_{q,p}=1.$ It follows that
$\lim_{p\rightarrow q}J_{q}(\widetilde{u}_{p})=\lim_{p\rightarrow q}\frac
{q-2}{2q}\left\Vert t_{q,p}u_{p}\right\Vert _{a,b}^{2}=\ell_{q}^{\phi},$ as claimed.
\end{proof}

\section{Minimizing sequences for the critical problem}

\label{sec:min_crit}In this section we analize the behavior of the minimizing
sequences for the problem (\ref{anisotropic_prob}) when $p$ is the critical
exponent $2_{n}^{\ast}=\frac{2n}{n-2}$. The solutions to the limit problem
(\ref{limprob})\ will play a crucial role in this analysis. We denote the
energy functional associated to (\ref{limprob}) by
\[
J_{\infty}(u):=\frac{1}{2}\left\Vert u\right\Vert ^{2}-\frac{1}{2^{\ast}%
}\left\vert u\right\vert _{2^{\ast}}^{2^{\ast}}%
\]
and, for any closed subgroup $K$ of $\Gamma$, we set
\begin{align*}
D^{1,2}(\mathbb{R}^{n})^{\phi\mid K}  &  :=\{u\in D^{1,2}(\mathbb{R}%
^{n}):u(\gamma z)=\phi(\gamma)u(z)\ \forall\gamma\in K,\,z\in\mathbb{R}%
^{n}\},\\
\mathcal{N}_{\infty}^{\phi\mid K}  &  :=\{u\in D^{1,2}(\mathbb{R}^{n}%
)^{\phi\mid K}:u\neq0,\text{ }\left\Vert u\right\Vert ^{2}=\left\vert
u\right\vert _{2^{\ast}}^{2^{\ast}}\},\\
\ell_{\infty}^{\phi\mid K}  &  :=\inf_{u\in\mathcal{N}_{\infty}^{\phi\mid K}%
}J_{\infty}(u).
\end{align*}
If $K=\Gamma$ we write $\mathcal{N}_{\infty}^{\phi}$ and $\ell_{\infty}^{\phi
}$ instead of $\mathcal{N}_{\infty}^{\phi\mid K}$ and $\ell_{\infty}^{\phi\mid
K}.$

Recall that the $\Gamma$-orbit of a point $x\in\mathbb{R}^{n}$ is the set
$\Gamma x:=\{\gamma x:\gamma\in\Gamma\},$ and its isotropy group is
$\Gamma_{x}:=\{\gamma\in\Gamma:\gamma x=x\}.$ Then, $\Gamma x\ $is $\Gamma
$-homeomorphic to the homogeneous space $\Gamma/\Gamma_{x}.$ In particular,
the cardinality of $\Gamma x$ is the index of $\Gamma_{x}$ in $\Gamma,$ which
is usually denoted by $\left\vert \Gamma/\Gamma_{x}\right\vert .$ If $\Gamma
x=\{x\}$ then $x$ is said to be a fixed point of $\Gamma.$ We denote
\[
\Theta^{\Gamma}:=\{x\in\Theta:x\text{ is a fixed point of }\Gamma\}.
\]

For simplicity, we will write $J_{\ast},$ $\mathcal{N}_{\ast}^{\phi}$ and
$\ell_{\ast}^{\phi}$ instead of $J_{2_{n}^{\ast}},$ $\mathcal{N}_{2_{n}^{\ast
}}^{\phi}$ and $\ell_{2_{n}^{\ast}}^{\phi}.$

\begin{theorem}
\label{thm:concentration}Let $(u_{k})$ be a sequence in $\mathcal{N}_{\ast
}^{\phi}$ such that $J_{\ast}(u_{k})\rightarrow\ell_{\ast}^{\phi}$. Then,
after passing a subsequence, one of the following two possibilities occurs:

\begin{enumerate}
\item $(u_{k})$ converges strongly in $D_{0}^{1,2}(\Theta)$ to a minimizer of
$J_{\ast}$ on $\mathcal{N}_{\ast}^{\phi}.$

\item There exist a closed subgroup $K$ of finite index in $\Gamma$, a
sequence $(\zeta_{k})$ in $\Theta$, a sequence $(\varepsilon_{k})$ in
$(0,\infty)$ and a nontrivial solution $\omega$ to the problem
\emph{(\ref{limprob})} with the following properties:

\begin{enumerate}
\item $\Gamma_{\zeta_{k}}=K$ \ for all $k\in\mathbb{N}$, and $\zeta
_{k}\rightarrow\zeta,\smallskip$

\item $\varepsilon_{k}^{-1}$\emph{dist}$(\zeta_{k},\partial\Theta
)\rightarrow\infty$ and $\varepsilon_{k}^{-1}|\alpha\zeta_{k}-\beta\zeta
_{k}|\rightarrow\infty$ for all $\alpha,\beta\in\Gamma$ with $\alpha^{-1}%
\beta\not \in K,\smallskip$

\item $\omega(\gamma z)=\phi(\gamma)\omega(z)$ for all $\gamma\in K,$
$z\in\mathbb{R}^{n},$ and \ $J_{\infty}(\omega)=\ell_{\infty}^{\phi\mid
K},\smallskip$

\item $\lim\limits_{k\rightarrow\infty}\left\Vert u_{k}-\sum\limits_{\left[
\gamma\right]  \in\Gamma/K}\phi(\gamma)\left(  \frac{a(\zeta)}{c(\zeta
)}\right)  ^{\frac{n-2}{4}}\varepsilon_{k}^{\frac{2-n}{2}}(\omega\circ
\gamma^{-1})(\frac{\text{ }\cdot\text{ }-\gamma\zeta_{k}}{\varepsilon_{k}%
})\right\Vert =0,\smallskip$

\item $\ell_{\ast}^{\phi}=\min\limits_{x\in\overline{\Theta}}\frac{a(x)^{n/2}%
}{c(x)^{(n-2)/2}}\left\vert \Gamma/\Gamma_{x}\right\vert \ell_{\infty}%
^{\phi|\Gamma_{x}}=\frac{a(\zeta)^{n/2}}{c(\zeta)^{(n-2)/2}}\left\vert
\Gamma/K\right\vert J_{\infty}(\omega).$
\end{enumerate}
\end{enumerate}
\end{theorem}

\begin{proof}
The proof is exactly the same as that of Theorem 2.5 in \cite{cf}, omitting
the first two lines.
\end{proof}

Let us state an interesting special case of Theorem \ref{thm:concentration}.

\begin{corollary}
\label{cor:concentration}Assume that every $\Gamma$-orbit in $\Theta$ is
either infinite or a fixed point. Let $(u_{k})$ be a sequence in
$\mathcal{N}_{\ast}^{\phi}$ such that $J_{\ast}(u_{k})\rightarrow\ell_{\ast
}^{\phi}$. Then, after passing a subsequence, one of the following statements
holds true:

\begin{enumerate}
\item $(u_{k})$ converges strongly in $D_{0}^{1,2}(\Theta)$ to a minimizer of
$J_{\ast}$ on $\mathcal{N}_{\ast}^{\phi}.$

\item There exist a sequence $(\zeta_{k})$ in $\Theta^{\Gamma},$ a sequence
$(\varepsilon_{k})$ in $(0,\infty)$ and a nontrivial $\phi$-equivariant
solution $\omega$ to the limit problem \emph{(\ref{limprob})} such that
$\zeta_{k}\rightarrow\zeta\in\overline{\Theta},$ $\varepsilon_{k}^{-1}%
$\emph{dist}$(\zeta_{k},\partial\Theta)\rightarrow\infty,$ $J_{\infty}%
(\omega)=\ell_{\infty}^{\phi},$
\[
\lim_{k\rightarrow\infty}\left\Vert u_{k}-\left(  \frac{a(\zeta)}{c(\zeta
)}\right)  ^{\frac{n-2}{4}}\varepsilon_{k}^{\frac{2-n}{2}}\omega\left(
\frac{\text{ }\cdot\text{ }-\zeta_{k}}{\varepsilon_{k}}\right)  \right\Vert
=0,
\]
and%
\[
\frac{a(\zeta)^{n/2}}{c(\zeta)^{(n-2)/2}}=\min\limits_{x\in\overline
{\Theta^{\Gamma}}}\frac{a(x)^{n/2}}{c(x)^{(n-2)/2}}.
\]

\end{enumerate}

In particular, if every $\Gamma$-orbit in $\Theta$ has positive dimension,
then \emph{(1)}\ must hold true.
\end{corollary}

\begin{proof}
Since the group $K=\Gamma_{\zeta_{k}},$ given by case (2) of Theorem
\ref{thm:concentration}, has finite index in $\Gamma$ and this index is the
cardinality of the $\Gamma$-orbit of $\zeta_{k},$ our assumption implies that
$K=\Gamma$ and $\zeta_{k}$ is a fixed point. So case (2) of Theorem
\ref{thm:concentration} reduces to case (2) of this corollary.
\end{proof}

Note that the functions $a$ and $c$ determine the location of the
concentration point $\zeta.$

It was shown in \cite[Theorem 2.3]{c}\ that, if $a=c=1$, $b=0$ and
$\Theta^{\Gamma}\neq\emptyset,$ then $\ell_{\ast}^{\phi}$ is not attained by
$J_{\ast}$ on $\mathcal{N}_{\ast}^{\phi}.$ So, if every $\Gamma$-orbit in
$\Theta\smallsetminus\Theta^{\Gamma}$ has positive dimension, statement (2) of
Corollary \ref{cor:concentration} must hold true.

In the following section we will state a nonexistence result which allows us
to obtain information on the limit profile of solutions to the problem
(\ref{prob}).

\section{Blow-up at the higher critical exponents}

\label{sec:blow-up}Throughout this section we will assume that $\Theta$ is a
$\Gamma$-invariant bounded smooth domain in $\mathbb{R}^{n}$ whose closure is
contained in $\mathbb{R}^{n-1}\times(0,\infty).$ Then, the points in
$\{0\}\times(0,\infty)$ must be fixed points of $\Gamma,$ so $\mathbb{R}%
^{n-1}\times\{0\}$ is $\Gamma$-invariant and we may regard $\Gamma$ as a
subgroup of $O(n-1).$ We will also assume that $\Theta\smallsetminus
\Theta^{\Gamma}$ and $\Theta^{\Gamma}$ are nonempty, and that every $\Gamma
$-orbit in $\Theta\smallsetminus\Theta^{\Gamma}$ has positive dimension. As
before, $\phi:\Gamma\rightarrow\mathbb{Z}_{2}$ will be a continuous
homomorphism which satisfies assumption (\ref{A}).

We set%
\[
\Theta_{\delta}:=\{x\in\Theta:\text{dist}(x,\Theta^{\Gamma})>\delta\}\text{
\ if }\delta>0,\text{\qquad and\qquad}\Theta_{0}:=\Theta,
\]
and we fix $\delta_{0}>0$ such that $\Theta_{\delta_{0}}\neq\emptyset$. \ For
$m\geq1$ and $\delta\in\lbrack0,\delta_{0}),$ we consider the problem%
\[
(\wp_{\delta,p}^{\#})\qquad%
\begin{cases}
-\,\text{div}(x_{n}^{m}\nabla u)+\lambda x_{n}^{m}u=x_{n}^{m}|u|^{p-2}u &
\text{in $\Theta_{\delta}$,}\\
u=0 & \text{on $\partial\Theta_{\delta}$.}\\
u(\gamma x)=\phi(\gamma)u(x), & \forall\gamma\in\Gamma\text{, \ }x\in
\Theta_{\delta},
\end{cases}
\]
where $x_{n}^{m}$ denotes the function $x=(x_{1},...,x_{n})\mapsto x_{n}^{m},$
and $\lambda\in(-\lambda_{1}^{\phi},\infty),$ with $\lambda_{1}^{\phi}$\ as
defined in (\ref{eq:lambda1}). Then, the operator $-\,$div$(x_{n}^{m}%
\nabla)+\lambda x_{n}^{m}$ is coercive in $D_{0}^{1,2}(\Theta)^{\phi}.$ So the
data of this problem satisfy all assumptions stated at the beginning of
Section \ref{sec:existence}.

Theorem \ref{thm:existence}\ asserts that the problem $(\wp_{\delta,p}^{\#})$
has a least energy solution $u_{\delta,p}$ if $\delta\in(0,\delta_{0})$ and
$p\in(2,2_{n-\mathfrak{d}}^{\ast}),$ where
\[
\mathfrak{d}:=\min\{\dim(\Gamma x):x\in\Theta\smallsetminus\Theta^{\Gamma}\}.
\]
Note that, by assumption, $\mathfrak{d}>0.$ On the other hand, for $\delta=0$
and $p=2_{n}^{\ast},$ the following nonexistence result was proved in
\cite{cf}.

\begin{theorem}
\label{thm:cf}If \emph{dist}$(\Theta^{\Gamma},\mathbb{R}^{n-1}\times
\{0\})=\,$\emph{dist}$(\Theta,\mathbb{R}^{n-1}\times\{0\}),$ then there exists
$\lambda_{\ast}\in(-\lambda_{1}^{\phi},0]$ such that, if $\lambda\in
(\lambda_{\ast},\infty)\cup\{0\},$ the critical problem $(\wp_{0,2_{n}^{\ast}%
}^{\#})$ does not have a least energy solution.

Moreover, $\lambda_{\ast}<0$ if $m\geq2.$
\end{theorem}

\begin{proof}
See Theorem 3.2 in \cite{cf}.
\end{proof}

For $\delta\in(0,\delta_{0})$ and $p\in(2,2_{n-\mathfrak{d}}^{\ast}),$ let
$J_{\delta,p}:D_{0}^{1,2}(\Theta_{\delta})^{\phi}\rightarrow\mathbb{R}$ be the
variational funcional and $\mathcal{N}_{\delta,p}^{\phi}$ be the Nehari
manifold associated to the problem $(\wp_{\delta,p}^{\#}),$ and set
\[
\ell_{\delta,p}^{\phi}:=\inf\{J_{\delta,p}(u):u\in\mathcal{N}_{\delta,p}%
^{\phi}\}.
\]
We write $J_{\ast},$ $\mathcal{N}_{\ast}^{\phi}$ and $\ell_{\ast}^{\phi}$ for
the variational functional, the Nehari manifold and the infimum associated to
the critical problem $(\wp_{0,2_{n}^{\ast}}^{\#})$ in the whole domain
$\Theta.$ Extending each function in $\mathcal{N}_{\delta,2_{n}^{\ast}}^{\phi
}$ by $0$ outside of $\Theta_{\delta},$ we have that $\mathcal{N}%
_{\delta,2_{n}^{\ast}}^{\phi}\subset\mathcal{N}_{\ast}^{\phi}$ and
$J_{\delta,2_{n}^{\ast}}(u)=J_{\ast}(u)$ for every $u\in\mathcal{N}%
_{\delta,2_{n}^{\ast}}^{\phi}.$ Hence, $\ell_{\ast}^{\phi}\leq\ell
_{\delta,2_{n}^{\ast}}^{\phi}.$

\begin{lemma}
\label{lem:delta_to_0}$\ell_{\delta,2_{n}^{\ast}}^{\phi}\rightarrow\ell_{\ast
}^{\phi}$ as $\delta\rightarrow0.$
\end{lemma}

\begin{proof}
Let $X:=(\mathbb{R}^{n})^{\Gamma}$ and $Y$ be its orthogonal complement in
$\mathbb{R}^{n}$. Since $\Theta\smallsetminus\Theta^{\Gamma}\neq\emptyset$ and
every $\Gamma$-orbit in $\Theta\smallsetminus\Theta^{\Gamma}$ has positive
dimension, we have that $\dim(Y)\geq2.$

We claim that there are radial functions $\chi_{k}\in\mathcal{C}_{c}^{\infty
}(Y)$ such that $\chi_{k}(y)=1$ if $\left\vert y\right\vert \leq\frac{1}{k},$
\begin{equation}
\lim_{k\rightarrow\infty}\int_{Y}\left\vert \chi_{k}\right\vert ^{2}%
=0\text{\qquad and\qquad}\lim_{k\rightarrow\infty}\int_{Y}\left\vert
\nabla\chi_{k}\right\vert ^{2}=0.\label{eq:lims}%
\end{equation}
To show this, we choose a radial function $g\in\mathcal{C}_{c}^{\infty}(Y)$
such that $g(y)=1$ if $\left\vert y\right\vert \leq1$ and $g(y)=0$ if
$\left\vert y\right\vert \geq2,$ and we set $g_{k}(y):=g(ky).$ Define
\[
\chi_{k}(y):=\frac{1}{\sigma_{k}}\sum_{j=1}^{k}\frac{g_{j}(y)}{j},\text{\qquad
where }\sigma_{k}:=\sum_{j=1}^{k}\frac{1}{j}.
\]
Clearly, $\chi_{k}(y)=1$ if $\left\vert y\right\vert \leq\frac{1}{k}$ and
$\chi_{k}(y)=0$ if $\left\vert y\right\vert \geq2.$ As $\dim(Y)\geq2,$ we have
that $\int_{Y}\left\vert \nabla g_{k}\right\vert ^{2}\leq\int_{Y}\left\vert
\nabla g\right\vert ^{2}.$ Hence, for some positive constant $C$,
\[
\int_{Y}\left\vert \nabla\chi_{k}\right\vert ^{2}\leq\frac{C}{\sigma_{k}^{2}%
}\sum_{j=1}^{k}\frac{1}{j^{2}}\rightarrow0\text{\qquad as }k\rightarrow\infty.
\]
Finally, as all functions $\chi_{k}$ are supported in the closed ball of
radius $2$ in $Y,$ the Poincar\'{e} inequality yields
\[
\int_{Y}\left\vert \chi_{k}\right\vert ^{2}\leq C\int_{Y}\left\vert \nabla
\chi_{k}\right\vert ^{2}\rightarrow0,
\]
and our claim is proved.

Given $\varepsilon>0$ we choose $\psi\in\mathcal{N}_{\ast}^{\phi}$ such that
$J_{\ast}(\psi)<\ell_{\ast}^{\phi}+\frac{\varepsilon}{2}.$ For $(x,y)\in
X\times Y,$ we define $\psi_{k}(x,y):=(1-\chi_{k}(y))\psi(x,y).$ Note that, as
$\chi_{k}$ is radial and $\psi$ is is $\phi$-equivariant, $\psi_{k}$ is also
$\phi$-equivariant$.$ Moreover, the identities (\ref{eq:lims}) easily imply
that $\psi_{k}\rightarrow\psi$ in $D_{0}^{1,2}(\Theta).$ So, for $k$ large
enough, there exists $t_{k}\in(0,\infty)$ such that $\widetilde{\psi}%
_{k}:=t_{k}\psi_{k}\in\mathcal{N}_{\ast}^{\phi}$ and $t_{k}\rightarrow1.$
Hence, $\widetilde{\psi}_{k}\rightarrow\psi$ in $D_{0}^{1,2}(\Theta),$ and we
may choose $k_{0}\in\mathbb{N}$ such that $J_{\ast}(\widetilde{\psi}_{k_{0}%
})<\ell_{\ast}^{\phi}+\varepsilon.$ Observe that supp$(\widetilde{\psi}%
_{k})=\,$supp$(\psi_{k})\subset\Theta_{\delta}$ if $\delta<\frac{1}{k}.$ So
$\widetilde{\psi}_{k}\in\mathcal{N}_{\delta,2_{n}^{\ast}}^{\phi}$ if
$\delta<\frac{1}{k}.$ It follows that
\[
\ell_{\ast}^{\phi}\leq\ell_{\delta,2_{n}^{\ast}}^{\phi}\leq J_{\delta
,2_{n}^{\ast}}(\widetilde{\psi}_{k_{0}})=J_{\ast}(\widetilde{\psi}_{k_{0}%
})<\ell_{\ast}^{\phi}+\varepsilon\text{\qquad}\forall\delta\in\left(
0,\frac{1}{k_{0}}\right)  .
\]
This finishes the proof.
\end{proof}

Set $N:=n+m$ and
\[
\Omega_{\delta}:=\{(y,z)\in\mathbb{R}^{n-1}\times\mathbb{R}^{m+1}%
:(y,\left\vert z\right\vert )\in\Theta_{\delta}\},\text{\qquad}\delta
\in\lbrack0,\delta_{0}).
\]
Note that $\Omega_{\delta}$ is $\left[  \Gamma\times O(m+1)\right]
$-invariant, i.e., $(\gamma y,\varrho z)\in\Omega_{\delta}$ for every
$(y,z)\in\Omega_{\delta},$ $\gamma\in\Gamma,$ $\varrho\in O(m+1).$ A
straighforward computation shows that $u_{\delta,p}$ is a least energy
solution to the problem $(\wp_{\delta,p}^{\#})$\ if and only if $v_{\delta
,p}(y,z):=u_{\delta,p}(y,\left\vert z\right\vert )$ is a least energy solution
to the problem%
\[
(\wp_{\delta,p})\qquad%
\begin{cases}
-\Delta v+\lambda v=|v|^{p-2}v & \text{in }\Omega\text{$_{\delta}$,}\\
v=0 & \text{on $\partial$}\Omega\text{$_{\delta}$,}\\
v(\gamma y,\varrho z)=\phi(\gamma)v(y,z), & \forall\gamma\in\Gamma,\text{
}\varrho\in O(m+1),\text{\ }(y,z)\in\Omega\text{$_{\delta}.$}%
\end{cases}
\]
Therefore, for every $\lambda\in(-\lambda_{1}^{\phi},\infty),$ $\delta
\in(0,\delta_{0})$ and $p\in(2,2_{n-\mathfrak{d}}^{\ast}),$ the problem
$(\wp_{\delta,p})$ has a least energy solution. The following results describe
its limit profile.

\begin{theorem}
\label{thm:critical_profile}For $\delta\in(0,\delta_{0})$ let $v_{\delta,\ast
}$ be a least energy solution to the problem $(\wp_{\delta,2_{N,m}^{\ast}}).$
Assume that
\[
\text{\emph{dist}}(\Theta^{\Gamma},\mathbb{R}^{n-1}\times
\{0\})=\text{\emph{dist}}(\Theta,\mathbb{R}^{n-1}\times\{0\}).
\]
Then, there exists $\lambda_{\ast}\leq0$ such that, if $\lambda\in
(\lambda_{\ast},\infty)\cup\{0\},$ there exist sequences $(\delta_{k})$ in
$(0,\delta_{0}),$ $(\varepsilon_{k})$ in $(0,\infty),$ $(\zeta_{k})$ in
$\Theta^{\Gamma},$ and a nontrivial solution $\omega$ to the limit problem
\emph{(\ref{limprob})} such that

\begin{enumerate}
\item[(i)] $\delta_{k}\rightarrow0,$ $\varepsilon_{k}^{-1}$\emph{dist}%
$(\zeta_{k},\partial\Theta)\rightarrow\infty,$ and $\zeta_{k}\rightarrow\zeta$
with%
\[
\text{\emph{dist}}(\zeta,\mathbb{R}^{n-1}\times\{0\})=\text{\emph{dist}%
}(\Theta,\mathbb{R}^{n-1}\times\{0\}),
\]

\item[(ii)] $\omega$ is $\phi$-equivariant and has minimal energy among all
nontrivial $\phi$-equivariant solutions to the problem \emph{(\ref{limprob})},

\item[(iii)] $v_{\delta_{k},\ast}=\widetilde{\omega}_{\varepsilon_{k}%
,\zeta_{k}}+o(1)$ in $D^{1,2}(\mathbb{R}^{N}),$ where%
\[
\widetilde{\omega}_{\varepsilon_{k},\zeta_{k}}(y,z):=\varepsilon_{k}%
^{(2-n)/2}\omega\left(  \frac{(y,\left\vert z\right\vert )-\zeta_{k}%
}{\varepsilon_{k}}\right)  .
\]

\end{enumerate}

Moreover, $\lambda_{\ast}<0$ if $m\geq2.$
\end{theorem}

\begin{proof}
Let $\lambda_{\ast}$ be the number given by Theorem \ref{thm:cf}. Fix
$\lambda\in(\lambda_{\ast},\infty)\cup\{0\},$ and let $u_{\delta,\ast}$ be the
least energy solution to the problem $(\wp_{\delta,2_{n}^{\ast}}^{\#})$ given
by $v_{\delta,\ast}(y,z)=u_{\delta,\ast}(y,\left\vert z\right\vert ).$ Choose
a sequence $\delta_{k}\rightarrow0$ and set $u_{k}:=u_{\delta_{k},\ast}$.
Then, $u_{k}\in\mathcal{N}_{\ast}^{\phi}$ and, by Lemma \ref{lem:delta_to_0},
$J_{\ast}(u_{k})\rightarrow\ell_{\ast}^{\phi}$. It follows from Corollary
\ref{cor:concentration} and Theorem \ref{thm:cf} that, after passing a
subsequence, there exist sequences $(\varepsilon_{k})$ in $(0,\infty)$ and
$(\zeta_{k})$ in $\Theta^{\Gamma},$ and a nontrivial $\phi$-equivariant
solution $\omega$ to the limit problem (\ref{limprob}) such that $\zeta
_{k}\rightarrow\zeta,$ $\varepsilon_{k}^{-1}$dist$(\zeta_{k},\partial
\Theta)\rightarrow\infty,$ $J_{\infty}(\omega)=\ell_{\infty}^{\phi},$%
\begin{equation}
\lim_{k\rightarrow\infty}\left\Vert u_{k}-\varepsilon_{k}^{\frac{2-n}{2}%
}\omega\left(  \frac{\text{ }\cdot\text{ }-\zeta_{k}}{\varepsilon_{k}}\right)
\right\Vert =0, \label{eq:profile_u}%
\end{equation}
and%
\[
\left[  \text{dist}(\zeta,\mathbb{R}^{n-1}\times\{0\})\right]  =\min
\limits_{x\in\overline{\Theta}}\left[  \text{dist}(x,\mathbb{R}^{n-1}%
\times\{0\})\right]  .
\]
Equation (\ref{eq:profile_u}) implies that $v_{\delta_{k},\ast}$ satisfies
(3). This concludes the proof.
\end{proof}

\begin{theorem}
\label{thm:supercritical_profile}For $\delta\in(0,\delta_{0})$ and
$p\in(2_{N,m}^{\ast},2_{N,m+\mathfrak{d}}^{\ast})$ let $v_{\delta,p}$ be a
least energy solution to the problem $(\wp_{\delta,p}).$ Assume that
\[
\text{\emph{dist}}(\Theta^{\Gamma},\mathbb{R}^{n-1}\times
\{0\})=\text{\emph{dist}}(\Theta,\mathbb{R}^{n-1}\times\{0\}).
\]
Then, there exists $\lambda_{\ast}\leq0$ such that, if $\lambda\in
(\lambda_{\ast},\infty)\cup\{0\},$ there exist sequences $(\delta_{k})$ in
$(0,\delta_{0}),$ $(\varepsilon_{k})$ in $(0,\infty)$, $(p_{k})$ in
$(2_{N,m}^{\ast},2_{N,m+\mathfrak{d}}^{\ast}),$ and $(\zeta_{k})$ in
$\Theta^{\Gamma},$ and a nontrivial solution $\omega$ to the limit problem
\emph{(\ref{limprob})} such that

\begin{enumerate}
\item[(i)] $\delta_{k}\rightarrow0,$ $p_{k}\rightarrow2_{N,m}^{\ast},$
$\varepsilon_{k}^{-1}$\emph{dist}$(\zeta_{k},\partial\Theta)\rightarrow
\infty,$ and $\zeta_{k}\rightarrow\zeta$ with%
\[
\text{\emph{dist}}(\zeta,\mathbb{R}^{n-1}\times\{0\})=\text{\emph{dist}%
}(\Theta,\mathbb{R}^{n-1}\times\{0\}),
\]

\item[(ii)] $\omega$ is $\phi$-equivariant and has minimal energy among all
nontrivial $\phi$-equivariant solutions to the problem \emph{(\ref{limprob})},

\item[(iii)] $v_{\delta_{k},p_{k}}=\widetilde{\omega}_{\varepsilon_{k}%
,\zeta_{k}}+o(1)$ in $D^{1,2}(\mathbb{R}^{N}),$ where%
\[
\widetilde{\omega}_{\varepsilon_{k},\zeta_{k}}(y,z):=\varepsilon_{k}%
^{(2-n)/2}\omega\left(  \frac{(y,\left\vert z\right\vert )-\zeta_{k}%
}{\varepsilon_{k}}\right)  .
\]

\end{enumerate}

Moreover, $\lambda_{\ast}<0$ if $m\geq2.$
\end{theorem}

\begin{proof}
Let $\lambda_{\ast}$ be the number given by Theorem \ref{thm:cf}. Fix
$\lambda\in(\lambda_{\ast},\infty)\cup\{0\}.$ Let $u_{\delta,p}$ be the least
energy solution to the problem $(\wp_{\delta,p}^{\#})$ given by $v_{\delta
,p}(y,z)=u_{\delta,p}(y,\left\vert z\right\vert )$ and let $t_{\delta,p}%
\in(0,\infty)$ be such that $\widetilde{u}_{\delta,p}:=t_{\delta,p}%
u_{\delta,p}\in\mathcal{N}_{\delta,2_{n}^{\ast}}^{\phi}\subset\mathcal{N}%
_{\ast}^{\phi}.$ Proposition \ref{prop:limits} and\ Lemma \ref{lem:delta_to_0}
allow us to choose $\delta_{k}\in(0,\delta_{0})$ and $p_{k}\in(2_{n}^{\ast
},2_{n-\mathfrak{d}}^{\ast})$ such that $\delta_{k}\rightarrow0,$
$p_{k}\rightarrow2_{n}^{\ast},$ and $J_{\ast}(\widetilde{u}_{k})\rightarrow
\ell_{\ast}^{\phi}$, where $\widetilde{u}_{k}:=\widetilde{u}_{\delta_{k}%
,p_{k}}.$ The rest of the proof is the same as that of Theorem
\ref{thm:critical_profile}
\end{proof}

Finally, we derive Theorems \ref{thm:main1}\ and \ref{thm:main2}\ from
Theorems \ref{thm:main_existence} and \ref{thm:supercritical_profile}.

\begin{proof}
[Proof of Theorem \ref{thm:main1}]Let $\Gamma:=O(n-1)$ and $\phi$ be the
trivial homomorphism $\phi\equiv1.$ Then, $B^{\Gamma}=B\cap\left[
\{0\}\times(0,\infty)\right]  .$ A $\phi$-equivariant function is simply a
$\Gamma$-invariant function and, as the standard bubble is radial, it is the
least energy $\Gamma$-invariant solution to the problem (\ref{limprob}), which
is unique up to translations and dilations. Since $\dim(\Gamma x)=n-2\geq1$
for every $x\in B\smallsetminus B^{\Gamma},$ applying Theorems
\ref{thm:main_existence} and \ref{thm:supercritical_profile} to $\Theta:=B$
with this group action we obtain Theorem \ref{thm:main1}.
\end{proof}

\begin{proof}
[Proof of Theorem \ref{thm:main2}]For $n\geq5,$ let $\Gamma$ be the subgroup
of $O(n-1)$ generated by $\{\mathrm{e}^{\mathrm{i}\vartheta},\alpha
,\tau:\vartheta\in\lbrack0,2\pi),$ $\alpha\in O(n-5)\}$ acting on a point
$y=(\eta,\xi)\in\mathbb{C}^{2}\times\mathbb{R}^{n-5}\equiv\mathbb{R}^{n-1},$
$\eta=(\eta_{1},\eta_{2})\in\mathbb{C}\times\mathbb{C}$, as
\[
\mathrm{e}^{\mathrm{i}\vartheta}y:=(\mathrm{e}^{\mathrm{i}\vartheta}\eta
,\xi),\text{\qquad}\alpha y:=(\eta,\alpha\xi),\text{\qquad}\tau y:=(-\overline
{\eta}_{2},\overline{\eta}_{1},\xi),
\]
and let $\phi$ be the homomorphism given by $\phi(\mathrm{e}^{\mathrm{i}%
\vartheta})=1=\phi(\alpha)$ and $\phi(\tau)=-1.$ Then, $B^{\Gamma}%
=B\cap\left[  \{0\}\times(0,\infty)\right]  .$ If $n=5$ then $\dim\left(
\Gamma y\right)  =1$ for every $y\in\mathbb{R}^{n-1}\smallsetminus\{0\},$
whereas for $n\geq6$ we have that%
\[
\dim\left(  \Gamma y\right)  =\left\{
\begin{array}
[c]{ll}%
n-5 & \text{if \ }\eta\neq0\text{ and }\xi\neq0,\\
1 & \text{if \ }\xi=0,\\
n-6 & \text{if \ }\eta=0.
\end{array}
\right.
\]
Therefore, if $n=5$ or $n\geq7,$ we have that $\dim(\Gamma x)\geq1$ for every
$x\in B\smallsetminus B^{\Gamma}.$ Notice that any point $x_{0}=(\eta,\xi)\in
B$ with\ $\eta\neq0$ satisfies condition (\ref{A})$.$ Hence, Theorem
\ref{thm:main2} follows from Theorems \ref{thm:main_existence} and
\ref{thm:supercritical_profile}.
\end{proof}


\begin{thebibliography}{99}                                                                                               %


\bibitem {acp}Ackermann, Nils; Clapp, M\'{o}nica; Pistoia, Angela: Boundary
clustered layers near the higher critical exponents. J. Differential Equations
254 (2013), no. 10, 4168--4193.

\bibitem {ar}Ambrosetti, Antonio; Rabinowitz, Paul H.: Dual variational
methods in critical point theory and applications. J. Functional Analysis 14
(1973), 349--381.

\bibitem {bcm}Bracho, Javier; Clapp, M\'{o}nica; Marzantowicz, Wac\l aw:
Symmetry breaking solutions of nonlinear elliptic systems. Topol. Methods
Nonlinear Anal. 26 (2005), no. 1, 189--201.

\bibitem {c}Clapp, M\'{o}nica: Entire nodal solutions to the pure critical
exponent problem arising from concentration. J. Differential Equations 261
(2016), no. 6, 3042--3060.

\bibitem {cf}Clapp, M\'{o}nica; Faya, Jorge: Concentration with a single sign
changing layer at the higher critical exponents. Advances in Nonlinear
Analysis, in press 2016. DOI 10.1515/anona-2016-0056.

\bibitem {cp}Clapp, M\'{o}nica; Pacella, Filomena: Existence and asymptotic
profile of nodal solutions to supercritical problems, Advanced Nonlinear
Studies, in press 2016. DOI 10.1515/ans-2016-6009.

\bibitem {cpi}Clapp, M\'{o}nica; Pistoia, Angela: Symmetries, Hopf fibrations
and supercritical elliptic problems. Mathematical Congress of the Americas,
1--12, Contemp. Math., 656, Amer. Math. Soc., Providence, RI, 2016.

\bibitem {dfm}del Pino, Manuel; Felmer, Patricio; Musso, Monica: Two-bubble
solutions in the super-critical Bahri-Coron's problem. Calc. Var. Partial
Differential Equations 16 (2003), no. 2, 113--145.

\bibitem {dfm2}del Pino, Manuel; Felmer, Patricio; Musso, Monica: Multi-bubble
solutions for slightly super-critical elliptic problems in domains with
symmetries. Bull. London Math. Soc. 35 (2003), no. 4, 513--521.

\bibitem {dmpp}del Pino, Manuel; Musso, Monica; Pacard, Frank; Pistoia,
Angela: Large energy entire solutions for the Yamabe equation. J. Differential
Equations 251 (2011), no. 9, 2568--2597.

\bibitem {dmp}del Pino, Manuel; Musso, Monica; Pacard, Frank: Bubbling along
boundary geodesics near the second critical exponent. J. Eur. Math. Soc.
(JEMS) 12 (2010), no. 6, 1553--1605.

\bibitem {hv}Hebey, Emmanuel; Vaugon, Michel: Sobolev spaces in the presence
of symmetries. J. Math. Pures Appl. (9) 76 (1997), no. 10, 859--881.

\bibitem {kp}Kim, Seunghyeok; Pistoia, Angela: Boundary towers of layers for
some supercritical problems. J. Differential Equations 255 (2013), no. 8, 2302--2339.

\bibitem {mp1}Molle, Riccardo; Passaseo, Donato: Positive solutions for
slightly super-critical elliptic equations in contractible domains. C. R.
Math. Acad. Sci. Paris 335 (2002), no. 5, 459--462.

\bibitem {mp2}Molle, Riccardo; Passaseo, Donato: Positive solutions of
slightly supercritical elliptic equations in symmetric domains. Ann. Inst. H.
Poincar\'{e} Anal. Non Lin\'{e}aire 21 (2004), no. 5, 639--656.

\bibitem {mpi}Musso, Monica; Pistoia, Angela: Persistence of Coron's solution
in nearly critical problems. Ann. Sc. Norm. Super. Pisa Cl. Sci. (5) 6 (2007),
no. 2, 331--357. Erratum: Ann. Sc. Norm. Super. Pisa Cl. Sci. (5) 8 (2009),
no. 1, 207--209.

\bibitem {mw}Musso, Monica; Wei, Juncheng: Sign-changing blowing-up solutions
for supercritical Bahri-Coron's problem. Calc. Var. Partial Differential
Equations 55 (2016), no. 1, Art. 1, 39 pp.

\bibitem {p1}Passaseo, Donato: Nonexistence results for elliptic problems with
supercritical nonlinearity in nontrivial domains. J. Funct. Anal. 114 (1993),
no. 1, 97--105.

\bibitem {p2}Passaseo, Donato: New nonexistence results for elliptic equations
with supercritical nonlinearity. Differential Integral Equations 8 (1995), no.
3, 577--586.

\bibitem {pr}Pistoia, Angela; Rey, Olivier: Multiplicity of solutions to the
supercritical Bahri-Coron's problem in pierced domains. Adv. Differential
Equations 11 (2006), no. 6, 647--666.
\end{thebibliography}
\end{document}